\documentclass[12pt]{article}

\usepackage{color}   
\usepackage{hyperref}
\hypersetup{
	colorlinks=true, 
	linktoc=all,     
	linkcolor=blue,  
	citecolor=red,
}
\usepackage{geometry}                
\usepackage{graphicx}
\usepackage{amssymb}
\usepackage{amsmath}
\usepackage{appendix}
\usepackage{tikz}
\usepackage[numbers]{natbib}

\usetikzlibrary{positioning}
\usetikzlibrary{arrows}
\newtheorem{remark}{Remark}
\newtheorem{definition}{Definition}[section]
\newtheorem{notation}{Notation}[section]

\newtheorem{properties}{Properties}[section]

\newtheorem{theorem}{Theorem}[section]

\newtheorem{lemma}{Lemma}[section]

\newtheorem{question}{Question}[section]
\newenvironment{proof}{{\noindent\it Proof}\quad}{\hfill $\square$\par} 

\numberwithin{figure}{section}
\numberwithin{equation}{section}

\usepackage{algorithm}
\usepackage{algorithmic}
\usepackage{bm}

\usepackage{stmaryrd}

\begin{document}
	
	\title{Deep Neural Networks and Finite Elements of \\Any Order on Arbitrary Dimensions}
	\author{Juncai He\footnotemark[1] \and  Jinchao Xu\footnotemark[1] \footnotemark[2]}
	\date{}                                           
	
	\maketitle
	\renewcommand{\thefootnote}{\fnsymbol{footnote}} 
	\footnotetext[1]{Computer, Electrical and Mathematical Science and Engineering Division, King Abdullah University of Science and Technology, Thuwal 23955, Saudi Arabia.} 
	\footnotetext[2]{Department of Mathematics, The Pennsylvania State University, University Park, PA 16802, USA.} 
	
	\begin{abstract}
		In this study, we establish that deep neural networks employing ReLU and ReLU$^2$ activation functions can effectively represent Lagrange finite element functions of any order on various simplicial meshes in arbitrary dimensions.
		We introduce two novel formulations for globally expressing the basis functions of Lagrange elements, tailored for both specific and arbitrary meshes.
		These formulations are based on a geometric decomposition of the elements, incorporating several insightful and essential properties of high-dimensional simplicial meshes, barycentric coordinate functions, and global basis functions of linear elements.
		This representation theory facilitates a natural approximation result for such deep neural networks. Our findings present the first demonstration of how deep neural networks can systematically generate general continuous piecewise polynomial functions on both specific or arbitrary simplicial meshes.
	\end{abstract}

	\section{Introduction}
	Deep neural networks (DNNs) with specific activation functions have attracted significant attention in scientific computing, particularly in the context of solving partial differential equations (PDEs)~\cite{lagaris1998artificial,yu2018deep,raissi2019physics,xu2020finite,siegel2023greedy}.
	To capitalize on the comprehensive mathematical findings and insights from traditional numerical analysis methods, exploring the expressivity properties of deep neural networks for classical methods is valuable. Specifically, we aim to study the relationships between DNNs and finite element functions~\cite{ciarlet2002finite}.
	For instance, it is recognized that DNNs with the rectified linear unit (ReLU) activation function~\cite{nair2010rectified}, defined as ${\rm ReLU}(t):=\max\{0,t\}$, represent a class of continuous piecewise linear (CPwL) functions. It has been shown that any CPwL function on $\mathbb{R}^d$ can be represented by a ReLU DNN with at most $\lceil \log_2(d+1) \rceil$ hidden layers~\cite{arora2018understanding}. However, as indicated in \cite{he2020relu}, such representations may require an extraordinarily large parameter count. By leveraging properties of linear finite element functions,  we developed a more succinct and efficient representation for linear finite element functions (CPwL functions with simplex meshes) in \cite{he2020relu}.
	Beyond representing general CPwL functions, several studies~\cite{montufar2014number,telgarsky2015representation,mhaskar2016deep,lu2017expressive,mehrabi2018bounds} have illustrated that ReLU networks with adequate depth and $\mathcal{O}(N)$ parameters can generate specific CPwL functions with segment counts exceeding any polynomial of $N$.
	These discoveries are pivotal in a body of research focusing on the approximation properties of very deep ReLU DNNs~\cite{yarotsky2017error,opschoor2020deep,montanelli2021deep,lu2021deep,marcati2023exponential}.
	A hierarchical basis viewpoint, originating from finite elements, was suggested in \cite{he2022relu} to enhance the understanding of these approximations.
	Nonetheless, as shown in \cite{daubechies2022nonlinear,he2023optimal}, such super-expressivity is confined to particular types of CPwL functions. More explicitly, whether ReLU DNNs with $\mathcal{O}(N)$ parameters of various widths and depths can accurately replicate any linear finite element function with $N$ degrees of freedom remains unanswered. Currently, positive evidence exists only for the interval $[0,1]$, as also demonstrated in \cite{daubechies2022nonlinear,he2023optimal}.
	
	On the other hand, the expressive capabilities of neural networks employing ${\rm ReLU}^k$ activation functions for $k \geq 2$ have also been the subject of research in works such as \cite{li2019better,xu2020finite,chen2022power,opschoor2022exponential}. These inquiries focus mainly on the representation of high-order splines within the interval $[0,1]$ or global polynomials on $\mathbb{R}^d$.
	In particular, \cite{he2023expressivity} provided a constructive proof with explicit bounds for weight values on how to uniformly express polynomials of arbitrary orders using ${\rm ReLU}^k$ activation functions for any $k \geq 2$.
	
	In summary of the preceding studies on the expressivity of ${\rm ReLU}$ and ${\rm ReLU}^k$ ($k \geq 2$) DNNs, we can draw the following two intuitive conclusions:
	\begin{enumerate}
		\item ${\rm ReLU}$ DNNs can represent linear finite element functions with any mesh.
		\item ${\rm ReLU}^k$ DNNs for $k \geq 2$ can represent global polynomials of any order.
	\end{enumerate}
	
	Consequently, a natural and fundamental question emerges:
	\begin{question}
		Is it possible to represent any continuous piecewise polynomials, particularly Lagrange finite element functions of any dimension and order, using a combination of ${\rm ReLU}$ and ${\rm ReLU}^k$ in DNNs?
	\end{question}
	
	A potential strategy for combining ${\rm ReLU}$ and ${\rm ReLU}^k$ in DNNs involves the so-called DNNs with learnable or arbitrary activation functions, which have been examined in~\cite{shen2021neural,liang2021reproducing,yarotsky2021elementary,chen2022power,longo2023rham,yang2023nearlyVC,yang2023nearlyAp}. Notably, the authors in \cite{chen2022power} suggest learning the activation function for each layer as a composition of ${\rm ReLU}^k$ for $k=0$ to some positive integer $n\ge 2$ to improve the efficiency of neural networks, where ${\rm ReLU}^0(t) = t$. 
	The work in \cite{longo2023rham} then demonstrated the use of ReLU, ReLU$^k$, and Heaviside activation functions to reproduce the lowest $H(\text{div})$ and $H(\text{curl})$ finite element functions in two or three dimensions. This study also showed how to represent discontinuous piecewise polynomials by using ReLU$^k$ to recover polynomials and Heaviside functions to represent the index functions for simplicial elements.
	More recently, \cite{yang2023nearlyVC,yang2023nearlyAp} explored the approximation properties of DNNs with arbitrary ${\rm ReLU}$ or ${\rm ReLU}^2$ activation functions. Nevertheless, most of these studies concentrate on the approximation capabilities of such DNNs or how this innovative architecture, in terms of activation functions, can enhance the performance of neural networks in practical applications. 
	In particular, the expressiveness of DNNs with learnable or arbitrary activation functions has not been thoroughly explored, particularly in terms of their connection to continuous piecewise polynomials (Lagrange finite element functions) of any dimension and order.

	To better present the main study of this paper, we introduce the following fully connected DNN architecture with learnable or arbitrary activation functions:
	\begin{equation}\label{eq:defdnn}
		\begin{cases}
			&u^0(\bm{x}) = \bm{x},\\
			&\left[u^j (\bm{x})\right]_i = \sigma^j_i \left( \left[W^j u^{j-1}(\bm{x}) + b^j\right]_i \right), \quad i=1,2,\cdots,n_j,~j=1,2,\cdots,L,\\
			&u(\bm{x}) = W^{L+1} u^L(\bm{x}) + b^{L+1},
		\end{cases}
	\end{equation}
	where $\sigma_{i}^j: \mathbb{R} \mapsto \mathbb{R}$ denotes the activation function of the $i$-th neuron in the $j$-th hidden layer for all $i=i=1,2,\cdots,n_j$ and $j=1,2,\cdots,L$.
	\begin{notation}
		Specifically, we introduce the following two types of DNNs:
		\begin{itemize}
			\item Standard DNNs with fixed ${\rm ReLU}^k$ activation function:
			\begin{equation}
				\Sigma^k_{n_{1:L}} := \left\{\text{DNNs defined by \eqref{eq:defdnn} with } \sigma_i^j = {\rm ReLU}^k,~ \forall~i, j \right\} 
			\end{equation}
			for any $k \in \mathbb{N}^+$. In particular, $\Sigma^1_{n_{1:L}}$ and $\Sigma^2_{n_{1:L}}$ denote the standard DNNs with ${\rm ReLU}$ and ${\rm ReLU}^2$ activation functions, respectively.
			\item DNNs with a special learnable activation function $\sigma_i^\ell$:
			\begin{equation}\label{eq:sigma}
				\Sigma^{1,2}_{n_{1:L}} := 
				\left\{ 
				\begin{aligned}
					&\text{DNNs defined by \eqref{eq:defdnn} with} \\
					&\sigma_i^j(t) = a^j_i {\rm ReLU}(t) + b^j_i {\rm ReLU}^2(t), 
				\end{aligned}
				\right\}
			\end{equation}
			where $t \in \mathbb{R}$ and $a^j_i, b^j_i \in \mathbb{R}$ are additional trainable parameters. Generally, this indicates that the activation function for each neuron is a combination of ${\rm ReLU}$ and ${\rm ReLU}^2$.
		\end{itemize}
	\end{notation}
	Consequently, we have
	\begin{equation}
		\Sigma^k_{n_{1:L}} \subseteq \Sigma^{1,2}_{n_{1:L}},
	\end{equation}
	for $k=1$ and $2$. Thus, a more specific question arises: Can $\Sigma^{1,2}_{n_{1:L}}$ accurately represent any Lagrange finite element functions?

	In this paper, we present a positive answer to the expressive capabilities of $\Sigma^{1,2}_{n_{1:L}}$ for representing Lagrange finite element functions of any dimension and order. This result is derived from two novel formulations of the global basis functions for Lagrange finite elements on specific or arbitrary simplicial meshes, incorporating a geometric decomposition of these elements. 
	We initially examine the representation for simplicial meshes with a convex constraint for the support set (patch) of basis functions. Central to this development are two foundational lemmas concerning high-dimensional simplicial meshes and the properties of barycentric functions. 
	For arbitrary simplicial meshes, we propose another uniform formulation using the global basis functions of the linear finite element space on the same mesh. The key of this construction lies in an intrinsic property in the multiplication structure of linear basis functions on a certain sub-simplex, seen as a generalization of the previous result for high-dimensional simplicial meshes with a convex constraint.
	By integrating these new formulations with the properties of $\Sigma^{1,2}_{n_{1:L}}$, we initially achieve two representations of the basis functions. Advancing this approach with a modified structure of $\Sigma^{1,2}_{n_{1:L}}$, we also capture representations of any Lagrange finite element functions with parameter levels same to their degrees of freedom.
	We will demonstrate that the representation of Lagrange finite element functions on simplicial meshes with a convex constraint is more concise and efficient. Furthermore, by examining a uniform simplex mesh over $[0,1]^d$ that satisfies the convex constraint and exploring the approximation characteristics of these functions, we obtain an additional approximation result for $\Sigma^{1,2}_{n_{1:L}}$. This offers a more intuitive and direct method for obtaining $W^{s,p}$ error estimates for $s=0,1$ and $p \in [2,\infty]$.

	The structure of this paper is organized as follows. In Section~\ref{sec:preliminaries}, we present some preliminaries about high-dimensional simplicial meshes and the geometric decomposition of Lagrange finite elements.
	In Section~\ref{sec:febasis}, we introduce these two new formulas to explicitly and globally define the basis functions of Lagrange finite elements for both specific and arbitrary simplicial meshes. In Section~\ref{sec:feDNN}, we show how DNNs can be employed to represent Lagrange finite element functions in any dimension and of any order. In Section~\ref{sec:application}, we show an approximation result that is a corollary of our representation insights. We draw some concluding remarks in Section~\ref{sec:conclusions}, reflecting on the implications of our findings.
	
	\section{Preliminaries}\label{sec:preliminaries}
	In this section, we introduce some preliminary results about simplicial meshes, Lagrange finite elements, and the geometric decomposition of these spaces~\cite{arnold2009geometric,chen2023geometric}. 
	
	\paragraph{Simplex, sub-simplex, simplicial mesh, and their properties in $\mathbb R^d$.}
	For any finite set $\left\{\bm{v}_1, \bm{v}_2, \cdots, \bm{v}_m \right\} \subset \mathbb{R}^d$, we denote the convex combination as
	\begin{equation}
		\text{Conv}\left(\bm{v}_1, \bm{v}_2, \cdots, \bm{v}_m \right) := \left\{\bm{x} = \sum_{i=1}^{m} a_i \bm{v}_i ~:~ \sum_{i=1}^{m} a_i = 1, ~ a_i \geq 0,~\forall i \right\}.
	\end{equation}
	Then, for any set $S\subseteq \mathbb R^d$, the convex hull of $S$ $\text{Conv}\left(S\right)$ is defined as the set of all convex combinations of points in $S$.
	A set $T \subset \mathbb{R}^d$ is called a simplex if it is the closed convex hull of $d+1$ affinely independent points $\{\bm{v}_0, \bm{v}_1, \cdots, \bm{v}_{d}\}$, i.e.,
	\begin{equation}
		T = \text{Conv}\left(\bm{v}_0, \bm{v}_1, \cdots, \bm{v}_{d} \right) = 
		\left\{ \bm x ~:~ \bm{x} = \sum_{i=0}^{d} \lambda_{i}(\bm x) \bm{v}_i \right\}.
	\end{equation}
	These points are the vertices of $T$, and these affine functions $\lambda_i(\bm x)$ are the barycentric coordinate functions. 
	For any simplex $T = {\rm Conv}\left(\bm{v}_0, \bm{v}_1, \cdots, \bm{v}_{d} \right) \subseteq \mathbb R^d$, we call $f = {\rm Conv}\left(\bm{v}^f_{0}, \bm{v}^f_{1}, \cdots, \bm{v}^f_{\ell} \right)$ as a $\ell$-dimensional sub-simplex of $T$ for any $\left\{\bm{v}^f_{0}, \bm{v}^f_{1}, \cdots, \bm{v}^f_{\ell}\right\} \subseteq \left\{\bm{v}_0, \bm{v}_1, \cdots, \bm{v}_{d} \right\}$ with some $0\le \ell \le d$. 
	Following \cite{arnold2009geometric}, let $\Delta(T)$ denote all the sub-simplices of $T$, while 
	$\Delta_\ell(T)$ denotes the set of sub-simplices of dimension $\ell$ for $\ell=0,1,\ldots,d$.
	
	From then on, we denote $\Omega \subset \mathbb R^d$ as a bounded polytope domain. We say that $\mathcal{T}$ is a simplicial mesh of $\Omega$ if it consists of $d$-dimensional simplices, that is, $\Omega = \bigcup_{T \in \mathcal{T}} T$, and the intersection of any two simplices is a face (a sub-simplex) common to these two simplices if the intersection is not empty.
	Correspondingly, for the simplicial mesh $\mathcal{T}$, we define
	\begin{equation}
		\Delta_\ell \left(\mathcal{T}\right) := \bigcup_{T \in \mathcal{T}} \Delta_\ell(T).
	\end{equation}
	That is, $\Delta_\ell(\mathcal{T})$ denotes the collection of all $\ell$-dimensional sub-simplices in $\mathcal{T}$. For example, $\Delta_0(\mathcal{T})$ represents the set of all vertices in $\mathcal{T}$, $\Delta_1(\mathcal{T})$ corresponds to the set of all edges in $\mathcal{T}$, and so on, up to $\Delta_d(\mathcal{T}) = \mathcal{T}$.
	
	\paragraph{Lagrange finite elements and global basis functions in classical definition.} 
	As introduced in~\cite{arnold2009geometric}, we first define the simplicial lattice, which will be used later in establishing the local basis functions for Lagrange finite elements with $k$-th order ($P_k$ element). 
	For any $k \geq 1$, the simplicial lattice of degree $k$ is given by
	\begin{equation}
		\mathbb T_k^d := \left\{ \bm \alpha \in \mathbb N^{d+1} ~:~ |\bm \alpha| = k
		\right\},
	\end{equation}
	where $\bm \alpha = (\alpha_0,\alpha_1,\ldots,\alpha_d)$ and $|\bm \alpha| = \sum_{i=0}^d \alpha_i$. For any simplex $T = \text{Conv}\left(\bm{v}_0, \bm{v}_1, \cdots, \bm{v}_{d} \right)$, we have the following interpolation points of $P_k$ elements defined by
	\begin{equation}
		\mathcal X_k(T) = \left\{\bm x_{\bm \alpha} = \frac{1}{k}\sum_{i=0}^d \alpha_i \bm v_i ~:~ \bm \alpha \in \mathbb T_k^d \right\}.
	\end{equation}
	$\mathcal X_k(T)$ can be interpreted as the geometric embedding of the algebraic set $\mathbb T_k^d$.
	Geometrically, $\mathcal X_k(T)$ represents the intersection points within $T$ of the hyperplanes defined by $\left\{ \bm{x} : k\lambda_i(\bm{x}) - j = 0\right\}$ for all $i=0, 1, \cdots, d$ and $j=0, 1, \cdots, k$.
	The cardinality of $\mathcal X_k(T)$ is $\binom{d+k}{k}$, which is also the cardinality of $\mathbb T_k^d$ and the dimension of $\mathbb P_k(T)$ the polynomial space on $T \subset \mathbb{R}^d$ of degree at most $k$.

	For any $\bm p \in \mathcal X_k(T)$, we denote $\varphi_{\bm p}(\bm{x}) \in \mathbb P_k(T)$ as the interpolation basis function corresponding to $\bm p $ such that
	\begin{equation}
		\varphi_{\bm p} (\bm p') = \begin{cases}
			1, & \text{if } \bm{p} = \bm{p}', \\
			0, & \text{if } \bm{p} \neq \bm{p}',
		\end{cases}
	\end{equation}
	for all $\bm p, \bm p' \in \mathcal X_k(T)$.
	The following theorem provides a uniform representation of the local basis function $\varphi_{\bm{p}}(\bm{x})$ on $T$ for any $\bm p \in \mathcal X_k(T)$.
	\begin{theorem}[\cite{burkardt2013finite,chen2023geometric}]
		For any $\bm p = \bm x_{\bm \alpha} \in \mathcal X_k(T)$ with $\bm \alpha \in \mathbb T_k^d$ and $\bm{x} \in T$, the basis function is given by
		\begin{equation}
			\varphi_{\bm{p}}(\bm{x}) = \frac{1}{\bm \alpha !} \prod_{i=0}^{d} \prod_{j=0}^{\alpha_i-1} (k\lambda_{i} (\bm{x}) - j),
		\end{equation}
		where $\bm \alpha ! = \prod_{i=0}^d \left(\alpha_i !\right)$.
	\end{theorem}

	\begin{definition}\label{def:indexpatch}
		For any simplicial mesh $\mathcal{T}$ of $\Omega$ with $N$ total elements (simplices) on $\Omega$, let us assume the simplices in $\mathcal{T}$ are ordered as $T_i$ for $i=1,\cdots,N$ and denote 
		\begin{equation}
			I_{S} := \{ i : S \subseteq T_i \} \quad \text{and} \quad P_{S} := \bigcup_{S \subseteq T \in \mathcal T} T = \bigcup_{i = I_{S}} T_i
		\end{equation}
		as the index set and patch set for any $S \subseteq \Omega$. With a slight abuse of notation, we denote $I_{\bm x}$ and $P_{\bm x}$ if $S = \{\bm x\}$ is a set of a single point.
	\end{definition}
	By aggregating $\mathcal X_T$ for all $T \in \mathcal{T}$, we obtain the set 
	\begin{equation}
		\mathcal X_k({\mathcal T}) := \bigcup_{T \in \mathcal{T}} \mathcal X_k(T).
	\end{equation}
	Therefore, the global basis function $\varphi_{\bm p}(\bm{x})$ for any $\bm{p} \in \mathcal X_k({\mathcal T})$ constructed by gluing $\left. \varphi_{\bm{p}}(\bm{x})\right|_{T_i}$ together for all $i \in I_{\bm{p}}$. 
	More precisely, $\varphi_{\bm{p}}(\bm{x}) = 0 $ if $\bm x \notin  P_{\bm p}$ and 
	for any $s \in I_{\bm{p}}$ with $\bm p = \bm x_{\alpha} \in \mathcal X_k({T_s})$, we have
	\begin{equation}\label{eq:basisn}
		\varphi_{\bm{p}}(\bm{x}) = \frac{1}{\bm \alpha! }\prod_{i=0}^{d} \prod_{j=0}^{\alpha_i-1} (k\lambda_{s,i} (\bm{x}) - j), \quad \forall \bm x \in T_s,
	\end{equation}
	where $\lambda_{s,i}(\bm x)$ are the barycentric coordinate functions corresponding to the $i$-th vertex in $T_s$ for $i=0,1,\cdots,d$. 
	Obviously, we have $\text{supp}(\varphi_{\bm{p}}) = P_{\bm p}$
	for any $\bm{p} \in \mathcal X_k({\mathcal T})$, where $\text{supp}(f(\bm{x})) := \overline{\{ \bm{x} : f(\bm{x}) \neq 0 \}}$ denotes the support set of $f(\bm{x})$. 
	
	\paragraph{Geometric Decomposition of $\mathcal X_k({\mathcal T})$.}
	As seen from the definition of the global basis function $\varphi_{\bm{p}}(\bm{x})$, it is necessary to determine the corresponding index $\bm{\alpha}$ such that $\bm{p} = \bm{x}_{\bm{\alpha}} \in \mathcal{X}_k(T_s)$ for different $T_s$ with $s \in I_{\bm{p}}$. The key observation is that some invariant properties of $\bm{\alpha}$ for $\bm{p}$ across different $T_s$ can be identified using a geometric decomposition of $\mathcal{X}_k(\mathcal{T})$. To better formulate this decomposition, for any two sets $A$ and $B$, we denote $A \oplus B = A \cup B$ if $A \cap B = \emptyset$.
	
	First, we define the local decomposition of $\mathcal{X}_k(T)$ for any $T \in \mathcal{T}$. Here, $\mathcal{X}_k^\ell(T)$ is defined by
	\begin{equation}
		\mathcal{X}_k^\ell(T) = \bigoplus_{f \in \Delta_\ell(T)} \mathcal{X}_k(\mathring{f}),
	\end{equation}
	where $\mathcal{X}_k(\mathring{f}) = \mathcal{X}_k(T) \cap \mathring{f}$. Here, $\mathring{f}$ denotes the interior of the $\ell$-dimensional sub-simplex $f$, and $\mathring{f} = f$ if $\ell = 0$.
	That is, $\mathcal{X}_k^\ell(T)$ represents those interpolation points in $\mathcal{X}_k(T)$ that are part of some $\ell$-dimensional sub-simplices in $\Delta_\ell(T)$ but do not belong to any sub-simplex in $\Delta_{\ell-1}(T)$. 
	Then, the decomposition of $\mathcal{X}_k(T)$ is defined as
	\begin{equation}\label{eq:decomXkT}
		\mathcal{X}_k(T) = \bigoplus_{\ell=0}^d \mathcal{X}_k^\ell(T).
	\end{equation}
	Figure~\ref{fig:2dp4} illustrates an example of the decomposition of $\mathcal{X}_4(T)$ on a 2D simplex $T$.

	\begin{figure}[H]
		\centering
		\begin{tikzpicture}[scale = 1.5 ]
			\draw (0,0) -- (4,0) -- (0,4) -- cycle;
			
			\fill [fill=blue, opacity=0.5] (0,0) circle (1.2ex);
			\fill [fill=blue, opacity=0.5] (0,4) circle (1.2ex);
			\fill [fill=blue, opacity=0.5] (4,0) circle (1.2ex);
			\fill [fill=green, opacity=0.5] (-0.15,0.85) rectangle (0.15,3.15);
			\fill [fill=green, opacity=0.5] (0.85,-0.15) rectangle (3.15,0.15);
			\fill [fill=green, opacity=0.5] (3,0.8) -- (3.23,1) -- (1,3.23) -- (0.8,3);
			\fill [fill=yellow, opacity=0.5] (0.88,0.88) -- (2.3,0.85) -- (0.85,2.3);
			
			\draw [dashed] (1,0) -- (1,3);
			\draw [dashed] (2,0) -- (2,2);
			\draw [dashed] (3,0) -- (3,1);
			\draw [dashed] (0,1) -- (3,1);
			\draw [dashed] (0,2) -- (2,2);
			\draw [dashed] (0,3) -- (1,3);
			\draw [dashed] (0,1) -- (1,0);
			\draw [dashed] (0,2) -- (2,0);
			\draw [dashed] (0,3) -- (3,0);
			\draw[black,fill=black] (0,0) circle (.4ex);
			\draw[black,fill=black] (0,1) circle (.4ex);
			\draw[black,fill=black] (0,2) circle (.4ex);
			\draw[black,fill=black] (0,3) circle (.4ex);
			\draw[black,fill=black] (0,4) circle (.4ex);
			\draw[black,fill=black] (1,0) circle (.4ex);
			\draw[black,fill=black] (1,1) circle (.4ex);
			\draw[black,fill=black] (1,2) circle (.4ex);
			\draw[black,fill=black] (1,3) circle (.4ex);
			\draw[black,fill=black] (2,0) circle (.4ex);
			\draw[black,fill=black] (2,1) circle (.4ex);
			\draw[black,fill=black] (2,2) circle (.4ex);
			\draw[black,fill=black] (3,0) circle (.4ex);
			\draw[black,fill=black] (3,1) circle (.4ex);
			\draw[black,fill=black] (4,0) circle (.4ex);
		\end{tikzpicture}
		\caption{
			An example of the decomposition of $\mathcal{X}_4(T)$ on a 2D simplex $T$ is illustrated. Here, $\mathcal{X}_4^0(T)$ denotes all the vertices of $T$ (marked in blue), $\mathcal{X}_4^1(T)$ represents all nodes on the interiors of edges (marked in green), and $\mathcal{X}_4^2(T)$ includes all nodes in the interior of the simplex $T$ (marked in yellow).
		}
		\label{fig:2dp4}
	\end{figure}
	Finally, by defining $\mathcal{X}_k^\ell(\mathcal{T}) = \bigcup_{T \in \mathcal T} \mathcal X_k^\ell(T) = \bigoplus_{f \in \Delta_\ell(\mathcal{T})} \mathcal{X}_k(\mathring{f})$, we have
	\begin{equation}
		\mathcal{X}_k(\mathcal{T}) = \bigoplus_{\ell=0}^d \mathcal{X}_k^\ell(\mathcal{T}) = \bigoplus_{\ell=0}^d \bigoplus_{f \in \Delta_\ell(\mathcal{T})} \mathcal{X}_k(\mathring{f}).
	\end{equation}
	One of the first properties given by the above decomposition is that
	\begin{equation}\label{eq:IpIf}
		I_{\bm{p}} = I_{f} \quad \text{and} \quad P_{\bm{p}} = P_{f}
	\end{equation}
	if $\bm{p} \in \mathcal{X}_k^\ell(\mathcal{T})$ for some $f \in \Delta_\ell(\mathcal{T})$ and $\bm{p} \in \mathring{f}$, or $\bm{p} \in \mathcal{X}_k(\mathring{f})$ for some $f \in \Delta_\ell(\mathcal{T})$.
	
	The geometric decomposition for Lagrange elements has been widely studied in \cite{arnold2006finite,arnold2009geometric,chen2021geometric,chen2023geometric}. However, these studies mainly focus on the local decomposition of polynomial spaces $\mathbb{P}_k(T)$. In the following, we will demonstrate how to use this decomposition and global basis for linear elements to derive two new formulations for global basis functions of Lagrange finite elements of any order on both specific and arbitrary simplicial meshes.
	
	\section{New formulations of global Lagrange element basis functions of any order on any dimensions}\label{sec:febasis}
	In this section, we introduce two innovative formulations of the global basis functions for Lagrange finite elements. The first formulation emerges from a geometric constraint for the mesh $\mathcal{T}$, coupled with two foundational insights pertaining to high-dimensional simplicial meshes and barycentric functions. The second formulation generalizes the result by removing the geometric constraint for the mesh $\mathcal{T}$.
	
	Given the decomposition of $\mathcal{X}_k(\mathcal{T})$ mentioned above, for any $\bm{p} \in \mathcal{X}_k^\ell(\mathcal{T})$ and $\bm{p} \in f = \text{Conv}(\bm{v}_0^f, \bm{v}_1^f, \ldots, \bm{v}_\ell^f)$, we first observe that there uniquely exist $\alpha_i \in \mathbb{N}^+$ for $i=0,1,\ldots,\ell$ with $\sum_{i=0}^\ell \alpha_i = k$ such that
	\begin{equation}
		\bm{p} = \frac{1}{k}\sum_{i=0}^\ell \alpha_i \bm{v}_i^f.
	\end{equation}
	To simplify the expression of our main theorem, we introduce the following definition.
	\begin{definition}\label{def:alpha}
		For any $0 \le \ell \le d$, $\bm{p} \in \mathcal{X}_k^\ell(\mathcal{T})$, and $\bm{p} \in \mathring{f}$ with $f \in \Delta_\ell(\mathcal{T})$, we always define the $\ell+1$ vertices in $f$ as the first $\ell+1$ vertices in $T_s$ for each $s \in I_{\bm{p}}$. 
	\end{definition}
	Consequently, for any $i \in I_{\bm{p}}$, we have $\bm{p} \in \mathcal{X}_k^\ell(T_i)$ and $\bm{p} = \bm{x}_{\bm{\alpha}}$ where 
	\begin{equation}
		\bm{\alpha} = (\alpha_0, \cdots, \alpha_\ell, 0, \ldots, 0) \in \mathbb{T}_k^d.
	\end{equation}
	That is, the index $\bm{\alpha}$ of $\bm{p} \in \mathcal{X}_k^\ell(\mathcal{T})$ depends only on the $f$ where $\bm p \in \mathring f$ and is an invariant on different $T_s$ for $s\in I_{\bm p}$ by applying Definition~\ref{def:alpha}.
	
	\paragraph{A novel formulation for global basis functions on vertex-convex meshes.}
	\begin{definition}For any $0\le \ell \le d$, we call a simplicial mesh $\mathcal T$ of $\Omega$ is ``$\ell$-convex'' if $P_f = \bigcup_{f  \subseteq T \in \mathcal T} T$ is convex for any $\ell$-dimensional sub-simplex $f \in \Delta_\ell(\mathcal  T)$. In particular, we call $\mathcal T$ ``vertex-convex'' if it is $0$-convex. Obviously, $\mathcal T$ is $d$-convex.
	\end{definition}
	
	\begin{lemma}\label{lem:vconvex}
		For any simplicial mesh $\mathcal{T}$ on $\Omega$ and $1\ < \ell < d$, $\mathcal T$ is $\ell$-convex if it is $(\ell-1)$-convex.
	\end{lemma}
	
	\begin{proof}
		For any $f = {\rm Conv}(\bm{v}_0^f, \bm v_1^f\ldots, \cdots \bm{v}_\ell^f) \in \Delta_\ell(\mathcal T)$ with some $0 < \ell < d$, let us define
		$$
		\widetilde f_i = {\rm Conv}(\bm{v}_0^f, \cdots, \bm v_{i-1}^f, \bm v_{i+1}^f, \cdots \bm{v}_\ell^f) \in \Delta_{\ell-1}(\mathcal T).
		$$
		Thus, $P_{\widetilde f_i}$ are convex for all $i=0,1,\ldots,\ell$.
		To complete the proof, we show that
		\begin{equation}
			P_f = \bigcap_{i=0}^\ell P_{\widetilde f_i},
		\end{equation}
		where $P_f$ and $P_{\widetilde f_i}$ are defined as in Definition~\ref{def:indexpatch}. 
		This will be sufficient to demonstrate that $P_f$ is convex, given the convexity of $P_{\widetilde f_i}$.
		
		First, for any $T \in \mathcal T$ such that $f \subseteq T$, we have $\widetilde f_i \subset T$ for all $i=0,1,\ldots, \ell$. That is, $T \subseteq P_{\widetilde f_i }$ for all $i=0,1,\ldots, \ell$, i.e.,
		\begin{equation}\label{eq:PfPi}
			P_f \subseteq \bigcap_{i=0}^\ell P_{\widetilde f_i}.
		\end{equation}
		Furthermore, by using the associative property of set algebra, we have
		\begin{equation*}
			\bigcap_{i=0}^\ell P_{\widetilde f_i} = \bigcap_{i=0}^\ell \bigcup_{j_i \in I_{\widetilde f_i}} T_{j_i} = \bigcup_{\substack{j_0 \in I_{\widetilde f_0} \\ \cdots \\ j_\ell \in I_{\widetilde f_0}}} \bigcap_{i=0}^\ell T_{j_i}.
		\end{equation*}
		Given the property of simplicial mesh, $\bigcap_{i=0}^\ell T_{j_i}$ can only be a sub-simplex or a simplex. Consequently, we have
		\begin{equation}\label{eq:decompPfi}
			\bigcap_{i=0}^\ell P_{\widetilde f_i} = \bigcup_{\ell=0}^d \mathcal S_\ell = \mathcal S_d \bigcup \mathcal S', 
		\end{equation}
		where $\mathcal S_\ell \subset \Delta_\ell (\mathcal T)$ denote the $\ell$-dimensional sub-simplices for all $\ell=0,1,\ldots,d$ and $\mathcal S' = \bigcup_{\ell=0}^{d-1}\mathcal S_\ell $.
		We first claim that $\mathcal S_d = P_f$. For any $T \in \mathcal T$ and $T \subseteq \mathcal S_d \subseteq \bigcap_{i=0}^\ell P_{\widetilde f_i}$, we have $\widetilde f_i \subseteq T$.
		Thus, $f \subset T$ and $\mathcal S_d = P_f$. 
		
		Now, we prove that $\mathcal S' = \emptyset$. If not, there exists $\bm x \in \mathcal S'$ and $\bm x \notin \mathcal S_d$. Hence, we have 
		$$
		{\rm dist}(\bm x, \mathcal S_d) := \inf_{\bm x' \in \mathcal S_d} \|\bm x-\bm x'\| = \delta > 0
		$$
		since $\mathcal S_d = P_f$ is a closed set. Then, for any $T \in \mathcal T$ and $T \subseteq \mathcal S_d$, we have that
		$$
		B_{\bm x} := B(\bm x, \delta/2) \bigcap {\rm Conv}\left(\{\bm x\} \cup \mathcal S_d \right) 
		$$
		has nonzero $d$-dimensional volume and $B_{\bm x} \bigcap \mathcal S_d = \emptyset$ where $B(\bm x, \delta/2)$ denotes the $d$-dimensional ball centered at $\bm x$ with radius $\delta/2$. 
		However, we have
		\begin{equation}
			B_{\bm x} \subseteq {\rm Conv}\left( \mathcal S' \bigcup \mathcal S_d \right) = 
			{\rm Conv}\left( \bigcap_{i=0}^\ell P_{\widetilde f_i} \right) = \bigcap_{i=0}^\ell P_{\widetilde f_i} = \mathcal S' \bigcup \mathcal S_d
		\end{equation}
		since $P_{\widetilde f_i}$ is convex for each $i=0,1\ldots,\ell$. 
		Thus, we have $B_{\bm x} \subset \mathcal S'$ since $B_{\bm x} \bigcap \mathcal S_d = \emptyset$. This is impossible since $\mathcal S'$ consists of low-dimensional sub-simplices which have only zero $d$-dimensional volume. 
		This finishes the proof.
	\end{proof}
	
	\begin{theorem}\label{thm:basisall-convex}
		For any vertex-convex simplicial mesh $\mathcal{T}$ of $\Omega$ for the $P_k$ element and $0 \le \ell \le d$, by taking Definition~\ref{def:alpha}, the basis function $\varphi_{\bm{p}}$ corresponding to the interpolation point $\bm{p} \in \mathcal X_k^\ell(\mathcal T)$ can be written as
		\begin{equation}\label{eq:varphip-convex}
			\varphi_{\bm{p}}(\bm{x}) = \frac{1}{\bm \alpha !} \min_{i' \in I_{\bm{p}}}\left\{\prod_{i=0}^\ell k{\rm ReLU}(\lambda_{i',i}(\bm{x})) \right\} \times 
			\prod_{i=0}^\ell \prod_{j=1}^{\alpha_i-1}\left(k\min_{i' \in I_{\bm{p}}}\left\{\lambda_{i',i} (\bm{x}) \right\} -j \right),
		\end{equation}
		where $\lambda_{i',i}(\bm x)$ denote the barycentric coordinates functions corresponding to the $i$-th vertex in the element $T_{i'}$ for all $i'\in I_{\bm p}$ and $i=0,1,\ldots,\ell$. With a slight abuse of notation, we assume $\lambda_{i',i}(\bm x)$ are defined on $\mathbb R^d$.
	\end{theorem}

	\begin{lemma}\label{lem:suppphik}
		With the same assumptions as in Theorem~\ref{eq:varphip-convex}, we have
		\begin{equation}\label{eq:suppphilambda}
			\text{supp}(\varphi_{\bm{p}}) = P_{\bm p} = P_{f} = \bigcap_{ i' \in I_{\bm{p}}} \bigcap_{i=0}^\ell \left\{ \bm{x} \in \Omega : \lambda_{i',i}(\bm{x}) \ge 0 \right\},
		\end{equation}
		where $\bm p \in \mathring f$ for some $f \in \Delta_\ell (\mathcal T)$.
	\end{lemma}
	\begin{proof}
		Given the definition of $\varphi_{\bm{n}}(\bm{x})$ in \eqref{eq:basisn} and properties in \eqref{eq:IpIf}, the first two identities hold. 
		For the last identity, we first observe the convexity of $\text{supp}(\varphi_{\bm{p}}) = P_f$ since $\mathcal T$ is vertex-convex.
		Moreover, the boundary set $\partial(\text{supp}(\varphi_{\bm{p}}))$ consists of faces 
		\begin{equation}\label{eq:faceij}
			F_{i',i}:= \left\{ \bm{x} \in T_{i'} : \lambda_{i',i}(\bm{x}) = 0 \right\},
		\end{equation}
		for all $i' \in I_{\bm{p}}$ and $i=0,1,\ldots,\ell$. Additionally, we have $\left\{ \bm{x} \in \Delta_i : \lambda_{i,j}(\bm{x}) \ge 0 \right\} \subseteq \text{supp}(\varphi_{\bm{n}})$. Thus, the $d$-dimensional polytope $P_f$ can be expressed as the intersection of half-spaces as in \eqref{eq:suppphilambda}.   
	\end{proof}
	
	Then, to prove the main theorem, let us establish the following lemmas regarding important properties of $\min_{i' \in I_{\bm{p}}}\left\{\lambda_{i',i} (\bm{x}) \right\}$ in $\varphi_{\bm{p}}(\bm{x})$ and $\min_{i' \in I_{\bm{p}}}\left\{\prod_{i=0}^\ell {\rm ReLU}\left(\lambda_{i',i}(\bm{x})\right) \right\}$, under the assumption that $\mathcal T$ is a vertex-convex mesh. 
	
	\begin{lemma}\label{lem:minlambda}
		For any vertex-convex simplicial mesh $\mathcal{T}$ of $\Omega$ for the $P_k$ element and $\bm{p} \in \mathcal X_k^\ell(\mathcal T)$ with some $0 \le \ell \le d$, by taking Definition~\ref{def:alpha}, we have
		\begin{equation}
			\min_{i' \in I_{\bm{p}}}\left\{\lambda_{i',i} (\bm{x}) \right\} = \lambda_{s,i}(\bm{x})
		\end{equation}
		for any $0\le i\le \ell$ and $\bm{x} \in T_s$ with $s \in I_{\bm{p}}$.
	\end{lemma}
	
	\begin{proof}
		Assuming that $\bm p \in \mathring f$ with some $ f = {\rm Conv}\left(\bm v_0^f, \bm v_1^f, \ldots, \bm v_\ell^f \right) \in \Delta_\ell(\mathcal T)$, for any $i=0,1,\ldots,\ell$, we have
		\begin{equation}
			\lambda_{i',i}(\bm{v}_i^f) = 1, \quad \forall i' \in I_{\bm{p}}.
		\end{equation}
		For any $\bm{x} \in T_s$ with some $s \in I_{\bm{p}}$ and $0\le i \le \ell$, given Lemma~\ref{lem:suppphik}, there is a unique point $\bm{y}_{\bm{x}} \in F_{s,i} = \left\{ \bm{x} \in T_s : \lambda_{s,i}(\bm{x}) = 0 \right\}$ such that there exists $t_{\bm{x}} \in [0,1]$ and
		\begin{equation}
			\bm{x} = t_{\bm{x}}\bm{v}_i^f + (1 - t_{\bm{x}})\bm{y}_{\bm{x}}.
		\end{equation}
		Since $\text{supp}(\varphi_{\bm{p}}) = P_{\bm p}$ is convex as shown in Lemma~\ref{lem:vconvex}, we have
		\begin{equation}
			\lambda_{i',i}(\bm{y}_{\bm{x}}) \ge 0, \quad \forall i' \in I_{\bm{p}},
		\end{equation}
		and in particular,
		\begin{equation}
			\lambda_{s,i}(\bm{y}_{\bm{x}}) = 0.
		\end{equation}
		Since $\lambda_{i,j}(\bm{x})$ are all affine maps, it follows that
		\begin{equation}
			\begin{aligned}
				\lambda_{i',i}(\bm{x}) = & \lambda_{i',i} \left(t_{\bm{x}}\bm{v}_i^f + (1 - t_{\bm{x}})\bm{y}_{\bm{x}}\right) \\
				= & t_{\bm{x}}\lambda_{i',i} (\bm{v}_i^f) +  (1 - t_{\bm{x}}) \lambda_{i',i} (\bm{y}_{\bm{x}}) \\
				\ge & t_{\bm{x}} \\
				= & t_{\bm{x}}\lambda_{s,i} (\bm{v}_i^f) +  (1 - t_{\bm{x}}) \lambda_{s,i} (\bm{y}_{\bm{x}}) \\
				= & \lambda_{s,i}(\bm{x})
			\end{aligned}
		\end{equation}
		for any $i \in I_{\bm{p}}$. Thus, the proof is complete.
	\end{proof}
	
	\begin{remark}
		Here, we note that a core identity used in \cite{he2020relu} and \cite{yarotsky2018optimal} to represent linear finite element functions using ReLU DNNs is essentially a special case of $\ell=0$ in the above lemma. An illustration of the proof is provided in Fig~\ref{fig:2dp2} for $\bm{p} \in \mathcal X_2^1(\mathcal T)$ within a 2D simplicial mesh.
	\end{remark}
	
	\begin{figure}[h]
		\centering
		\begin{tikzpicture}[scale = 2 ]
			\draw[gray, step = 2cm] (0, 0) grid (2, 2);
			\node[anchor = north] at (1, 0) {$\lambda_{s,0}=0$};
			\node[anchor = south, rotate=90] at (0,1) {$\lambda_{s,1}=0$};
			\node[anchor = south] at (1,2) {$\lambda_{s',1}=0$};
			\node[anchor = north, rotate=90] at (2,1) {$\lambda_{s',0}=0$};
			\node[anchor = south east] at (0,2) {$\bm v_0^f$};
			\node[anchor = north west] at (2,0) {$\bm v_1^f$};
			\node[anchor = west] at (1,1) {$\bm p$};
			\node at (0.7, 0.7) {$T_s$};
			\node at (1.4, 1.4) {$T_{s'}$};
			\draw (0,2) -- (2,0);
			\draw[blue, thick] (0,0) -- (2,0) -- (2,2) -- (0,2) -- (0,0);
			\draw[black,fill=black] (1,1) circle (.3ex);
			\draw[black,fill=black] (0,0) circle (.2ex);
			\draw[black,fill=black] (0,1) circle (.2ex);
			\draw[black,fill=black] (0,2) circle (.2ex);
			\draw[black,fill=black] (1,2) circle (.2ex);
			\draw[black,fill=black] (1,0) circle (.2ex);
			\draw[black,fill=black] (2,0) circle (.2ex);
			\draw[black,fill=black] (2,1) circle (.2ex);
			\draw[black,fill=black] (2,2) circle (.2ex);
			
			\node[anchor = south west] at (0.225,0.975) {$\bm x$};
			\node[anchor = north] at (0.65,0.25) {$\bm y_{\bm x}$};
			\draw[black,fill=black] (0.25,1) circle (.1ex);
			\draw[black,fill=black] (0.5,0) circle (.1ex);
			\draw[red, thick] (0,2) -- (0.25,1) -- (0.5,0);
			
		\end{tikzpicture}
		\caption{
			An example of $\bm p \in \mathcal X_2^1(\mathcal T)$ on 2D, where $I_{\bm p} = \{s ,s'\}$ and $\bm p \in f = {\rm Conv}(\bm v_0^f, \bm v_1^f)$. We have $\min\left\{\lambda_{s,i}(\bm x), \lambda_{s',i}(\bm x)\right\} = \lambda_{s,i}(\bm x)$ for any $\bm x \in T_s$ and $i=0,1$.}
		\label{fig:2dp2}
	\end{figure}

	\begin{lemma}\label{lem:minprodlambda}
		For any vertex-convex simplicial mesh $\mathcal{T}$ of $\Omega$ for the $P_k$ element and $\bm{p} \in \mathcal X_k^\ell(\mathcal T)$ with some $0 \le \ell \le d$, by taking Definition~\ref{def:alpha}, we have
		\begin{equation}
			\min_{i' \in I_{\bm{p}}}\left\{\prod_{i=0}^\ell {\rm ReLU}\left(\lambda_{i',i}(\bm{x})\right) \right\} = \begin{cases}
				\prod_{i=0}^\ell \lambda_{s,i}(\bm{x}) & \text{if } \bm{x} \in T_s \text{ and } s \in I_{\bm{p}} \\
				0 & \text{if } \bm{x} \notin P_{\bm p}.
			\end{cases}
		\end{equation}
	\end{lemma}
	\begin{proof}
		Given Lemma~\ref{lem:suppphik}, we observe that $\prod_{i=0}^\ell {\rm ReLU}\left(\lambda_{i',i}(\bm{x})\right) \ge 0$ for all $i' \in I_{\bm p}$. Then, if $\bm{x} \notin  P_{\bm p}$, there exists a $\lambda_{i,i}$ for some $i' \in I_{\bm{p}}$ and $0\le i \le \ell$ such that $\lambda_{i',i}(\bm{x}) < 0$. Therefore, $\min_{i' \in I_{\bm{p}}}\left\{\prod_{i=0}^\ell {\rm ReLU}\left(\lambda_{i',i}(\bm{x})\right) \right\} = 0$ if $\bm{x} \notin P_{\bm p}$.
		
		For any $\bm{x} \in \text{supp}(\varphi_{\bm{p}}) = P_{\bm p}$, there exists $s \in I_{\bm{p}}$ such that $\bm{x} \in T_s$. Moreover, we have $\lambda_{i',i}(\bm{x}) \ge 0$ for all $i' \in I_{\bm{p}}$ and $i = 0, 1, \ldots, \ell$ due to the connection in \eqref{eq:suppphilambda}. Combined with Lemma~\ref{lem:minlambda}, it follows that
		\begin{equation}
			\min_{i' \in I_{\bm{p}}}\left\{\prod_{i=0}^\ell {\rm ReLU}\left(\lambda_{i',i}(\bm{x})\right) \right\} = \prod_{i=0}^\ell \lambda_{s,i}(\bm{x})
		\end{equation}
		when $\bm{x} \in T_s$ for some $s \in I_{\bm{p}}$. This completes the proof.
	\end{proof}
	
	\paragraph{Proof of Theorem~\ref{thm:basisall-convex}} 
	\begin{proof}
		Given Lemma~\ref{lem:minlambda} and Lemma~\ref{lem:minprodlambda}, for any $\bm x \in T_s$ with some $s \in I_{\bm p}$, we have
		\begin{equation}
			\begin{aligned}
				&\frac{1}{\bm \alpha !} \min_{i' \in I_{\bm p}}\left\{\prod_{i=0}^\ell {\rm ReLU}\left(\lambda_{i',i}(\bm x)\right) \right\} \times 
				\prod_{i=0}^\ell\prod_{j=1}^{\alpha_i-1}\left(k\min_{i' \in I_{\bm p}}\left\{\lambda_{i',i} (\bm x) \right\} -j\right) \\
				=&\frac{1}{\bm \alpha !} \left\{\prod_{i=0}^\ell {\rm ReLU}\left(\lambda_{s,i}(\bm x)\right) \right\} \times 
				\prod_{i=0}^\ell\prod_{j=1}^{\alpha_i-1}\left(k \lambda_{s,i} (\bm x) -j\right) \\
				=&\frac{1}{\bm \alpha !} \prod_{i=0}^\ell\prod_{j=0}^{\alpha_i-1}\left(k\left\{\lambda_{s,i} (\bm x) \right\} -j \right) =\frac{1}{\bm \alpha !} \prod_{i=0}^{d}\prod_{j=0}^{\alpha_i-1}\left(k\left\{\lambda_{s,i} (\bm x) \right\} -j \right).
			\end{aligned}
		\end{equation}
	\end{proof}

	\paragraph{A novel formulation for global basis functions on arbitrary simplicial meshes.}
	In this section, we propose another new formulation for the global basis functions of Lagrange elements on arbitrary simplicial meshes. We construct the piecewise polynomial basis functions by using the multiplication of piecewise linear functions, which are the basis functions of linear finite elements on the same simplicial mesh. The most crucial step in this construction is based on a generalization of Lemma~\ref{lem:vconvex}.
	
	We start our construction with the following observation. For $\ell=0$, i.e., $\bm p \in \mathcal X_k^0(\mathcal T) = \Delta_0(\mathcal T)$, we have
	\begin{equation*}
		\varphi_{\bm{p}}(\bm{x}) = \frac{1}{\alpha_0 !} 
		\prod_{j=0}^{\alpha_0 -1}\left(k \lambda_{s,0}(\bm{x}) -j \right)
	\end{equation*}
	for  any $\bm x \in T_s$ for $s \in I_{\bm p}$. 
	\begin{definition}
		For any simplicial mesh $\mathcal T$ of $\Omega \subset \mathbb R^d$, let $\phi_{\bm v}$ denote the piecewise linear basis function at $\bm v \in \mathcal X_k^0({\mathcal T}) = \Delta_0(\mathcal T)$, i.e., 
		\begin{equation*}
			\phi_{\bm v}(\bm x) = \begin{cases}
				\lambda_{s,j} \quad &\text{if } \bm x \in T_s, \text{ for some } s \in I_{\bm v} \text{ and } \bm v \text{ is the } j\text{-th vertex in } T_i  \\
				0\quad &\text{if } x \notin \bigcup_{ s \in I_{\bm v}} T_s = P_{\bm v}.
			\end{cases}
		\end{equation*}
	\end{definition}
	Consequently, the aforementioned global basis function $\varphi_{\bm{p}}(\bm{x})$ for any $\bm p \in \mathcal X_k^0(\mathcal T)$ can be represented as:
	\begin{equation}\label{eq:basis0}
		\varphi_{\bm{p}}(\bm{x}) = \frac{1}{\alpha_0 !} 
		\prod_{j=0}^{\alpha_0 -1}\left(k \phi_{\bm p}(\bm{x}) -j \right), \quad \forall \bm x \in \Omega.
	\end{equation}
	Here, we notice that $\varphi_{\bm{p}}(\bm{x}) = 0$ for any $\bm x \notin P_{\bm v}$ since $\prod_{j=0}^{\alpha_0 -1}\left(k \phi_{\bm p}(\bm{x}) -j \right) =k \phi_{\bm p}(\bm{x}) \times \prod_{j=1}^{\alpha_0 -1}\left(k \phi_{\bm p}(\bm{x}) -j \right)$.
	
	A natural question arises as to whether we can generalize \eqref{eq:basis0} to arbitrary $\ell > 0$. A main observation of this effort is summarized as the following lemma, which can be understood as a generalization of Lemma~\ref{lem:vconvex}.
	\begin{lemma}\label{lem:linearbasis}
		For any simplicial mesh $\mathcal T$ and $f \in \Delta_\ell(\mathcal T)$ with $0\le \ell \le d$ and $f = {\rm Conv}(\bm v_0^f, \bm v_1^f \ldots, \bm v_\ell^f)$, we have
		\begin{equation*}
			\prod_{i=0}^\ell \phi_{\bm v_i^f}(\bm x) = 0, \quad \text{if} \quad \bm x \notin \bigcup_{f \subseteq T \in \mathcal T} T = P_f.
		\end{equation*}
	\end{lemma}
	\begin{proof}
		This proof mainly follows the proof of Lemma~\ref{lem:vconvex}. First, we have 
		$$
		\prod_{i=0}^\ell \phi_{\bm v_i^f}(\bm x) = 0 \quad \text{ if }\quad  \bm x \notin \bigcap_{i=0}^\ell P_{v_i^f}.
		$$
		Then, similar to the proof of Lemma~\ref{lem:vconvex}, we have the following decomposition
		$$
		\bigcap_{i=0}^\ell P_{\bm v_i^f} = \bigcup_{\ell=0}^d \mathcal S_\ell = \mathcal S_d \bigcup \mathcal S'.
		$$
		If $\mathcal S'$ is nonempty, a key observation here is that 
		\begin{equation}\label{eq:propertyS}
			f \not\subset f' \quad \text{ for any } \quad f' \in \Delta_{\ell'}(\mathcal T) \text{ and } f' \subseteq \mathcal S'.
		\end{equation}
		
		Thus, we can complete the proof by showing that $\prod_{i=0}^\ell \phi_{\bm v_i^f}(\bm x) = 0$ for all $\bm x \in f' \subseteq \mathcal S'$ for any $f' \in \Delta_{\ell'}(\mathcal T)$ and $f' \subset \mathcal S'$. For this purpose, we claim 
		\begin{equation}\label{eq:f'boundary}
			f' \subset \bigcup_{i=0}^\ell \left(\partial P_{\bm v_i^f}\right)
		\end{equation}
		for any $f' \subseteq \mathcal S'$ with $f' \in \Delta_{\ell'}(\mathcal T)$. 
		If not, there exists $\bm x \in f' \subseteq \mathcal S'$ for some $f'\in \Delta_{\ell'}(\mathcal T)$ with $\ell' < d$ such that $\bm x \in \bigcap_{i=0}^\ell \mathring P_{\bm v_i^f}$. 
		That is, there exists a $d$-dimensional ball $B(\bm x, r)$ centered at $\bm x$ with radius $r>0$ such that $B(\bm x, r) \subset \mathring P_{\{\bm v_{i}^f\}}$ for all $i=0,1,\ldots,\ell$. Hence, we have $f' \bigcap B(\bm x, r) \subset \mathring P_{\{\bm v_{i}^f\}}$. This implies that $f' \subset P_{\bm v_{i}^f}$ since $\mathring{f'} \bigcap \mathring P_{\bm v_{i}^f} \neq \emptyset$ for all $i=0,1,\ldots,\ell$. 
		Therefore, there exists $T_{\bm x, i} \in \mathcal T$ with $T_{\bm x, i} \subset P_{\{\bm v_{i}^f\}}$ such that $\bm x \in f' \subset \partial T_{\bm x, i}$ for all $i=0,1,\ldots,\ell$. 
		Furthermore, noting that $F_{\bm x, i} \subset \partial P_{\bm v_{i}^f}$ where $F_{\bm x, i}$ denotes the $(d-1)$-hyperplane face corresponding to vertex $\bm v_i^f$ in $T_{\bm x, i}$, we have $\bm x \notin F_{\bm x, i} \in \Delta_{d-1}(T_{\bm x, i})$. Thus, any sub-simplex that contains $\bm x$ must also contain vertex $\bm v_i^f$. That is, $\bm v_i^f \in f'$ for all $i = 0,1,\ldots,\ell$ since $\bm x \in f'$ is a $\ell'$-dimensional sub-simplex. Therefore, $f = {\rm Conv}(\bm v_0^f, \bm v_1^f \ldots, \bm v_\ell^f) \subseteq f'$. This contradicts the observation \eqref{eq:propertyS}. Thus, the claim \eqref{eq:f'boundary} is correct. Consequently, $\prod_{i=0}^\ell \phi_{\bm v_i^f}(\bm x) = 0$ for all $\bm x \in \mathcal S'$ since $\phi_{\bm v_i^f}(\bm x) = 0$ for all $\bm x \in \partial P_{\bm v_i^f}$ for any $i=0,1,\cdots,\ell$. This finishes the proof.
	\end{proof}

	\begin{remark}
		Here, we note that $\mathcal S'$ can be nonempty if $\mathcal T$ is not a vertex-convex mesh. Fig.~\ref{fig:2dnonconvexmesh} demonstrates an example on $[0,1]^2$.
	\end{remark}
	\begin{figure}[h]
		\centering
		\begin{tikzpicture}[scale = 2 ]
			\node[anchor = east] at (1,1) {$\bm v_0^f$};
			\node[anchor = west] at (2,1) {$\bm v_1^f$};
			\node[anchor = south] at (1.5,0.85) {$f$};
			\node[anchor = west] at (1.42,1.7) {$f'$};
			\draw[blue, thick] (0,0) -- (3,0) -- (3,3) -- (0,3) -- (0,0);
			\draw[blue, thick] (0,0) -- (1,1) -- (2,1) -- (3,0);
			\draw[blue, thick] (1,1) -- (1.5,0) -- (2,1);
			\draw[blue, thick] (1,1) -- (1.5,1.5) -- (2,1);
			\draw[blue, thick] (1,1) -- (1.5,2) -- (2,1);
			\draw[blue, thick] (1,1) -- (0,3);
			\draw[blue, thick] (2,1) -- (3,3);
			\draw[red, thick] (1.5,1.5) -- (1.5,2);
			\draw[blue, thick] (0,3) -- (1.5,2) -- (3,3);
			\draw[black,fill=black] (1,1) circle (.1ex);
			\draw[black,fill=black] (0,0) circle (.1ex);
			\draw[black,fill=black] (1.5,0) circle (.1ex);
			\draw[black,fill=black] (1.5,1.5) circle (.1ex);
			\draw[black,fill=black] (1.5,2) circle (.1ex);
			\draw[black,fill=black] (2,1) circle (.1ex);
			\draw[black,fill=black] (3,0) circle (.1ex);
			\draw[black,fill=black] (3,3) circle (.1ex);
			\draw[black,fill=black] (0,3) circle (.1ex);
			
		\end{tikzpicture}
		\caption{
			An example on 2D showing that $\mathcal S'$ is nonempty. In this example, $P_{\bm v_0^f}$ and $P_{\bm v_1^f}$ are not convex, and $S' = f' \in \Delta_1(\mathcal T)$ is a 1-dimensional sub-simplex marked as the red segment in the diagram. Here, we have $f' \subseteq \partial P_{\bm v_0^f} \cap \partial P_{\bm v_1^f}$.}
		\label{fig:2dnonconvexmesh}
	\end{figure}
	
	\begin{remark}
		For any $T \in \mathcal{T}$ and $T \subseteq P_f$, we note that $b(\bm{x}) := \left.\prod_{i=0}^\ell \phi_{\bm{v}_i^f}(\bm{x}) \right|_{T} = \lambda_0(\bm{x})\lambda_1(\bm{x})\cdots\lambda_\ell(\bm{x})$ (by using Definition~\ref{def:alpha}) is the local bubble function as studied in the geometric decomposition of Lagrange finite elements in~\cite{arnold2009geometric,chen2023geometric}. The main property of this local bubble function is $b(\bm{x}) = 0$ for any $\bm{x} \in \partial f$. By gluing these local barycentric coordinate functions $\lambda_{i}(\bm{x})$ into linear basis functions $\phi_{\bm{v}_i^f}(\bm{x})$, the above lemma presents an important property similar to the local bubble function.
	\end{remark}

	Now, we have the following main theorem about the new global formula for Lagrange basis functions on an arbitrary simplicial mesh.
	\begin{theorem}\label{thm:basisall-noncovex}
		For any simplicial mesh $\mathcal{T}$ of $\Omega$ for the $P_k$ element and interpolation point $\bm p \in \mathcal X_k^\ell({\mathcal T})$ for any $\ell=0,1,\ldots,d$, by taking Definition~\ref{def:alpha}, the basis function $\varphi_{\bm{p}}$ can be written as
		\begin{equation}\label{eq:varphip-noncovex}
			\varphi_{\bm{p}}(\bm{x}) = \frac{1}{\bm \alpha !} 
			\prod_{i=0}^\ell\prod_{j=0}^{\alpha_i -1}\left(k \phi_{\bm v_i^f}(\bm{x}) -j \right), \quad \forall \bm x \in \Omega,
		\end{equation}
		where $f = {\rm Conv}(\bm v_0^f, \bm v_1^f, \ldots, \bm v_\ell^f)$ with $\bm p \in \mathring f$ and $\phi_{\bm v_i^f}$ is the piecewise linear basis function at $\bm v_i^f \in \mathcal X_k^0({\mathcal T})$ regarding the simplicial mesh $\mathcal T$.
	\end{theorem}
	
	\paragraph{Proof of Theorem~\ref{thm:basisall-noncovex}}
	\begin{proof}
		First, we can rewrite \eqref{eq:varphip-noncovex} to
		\begin{equation}
			\varphi_{\bm{p}}(\bm{x}) = \frac{1}{\bm \alpha !}  \prod_{i=0}^\ell k\phi_{\bm v_i^f}(\bm{x}) \times 
			\prod_{i=0}^\ell\prod_{j=1}^{\alpha_i -1}\left(k \phi_{\bm v_i^f}(\bm{x}) -j \right), \quad \forall \bm x \in \Omega
		\end{equation}
		since $\alpha_i \ge 1$ for all $i=0,1,\ldots,\ell$. Given Lemma~\ref{lem:linearbasis}, we have 
		$\varphi_{\bm{p}}(\bm{x}) = 0$ for any $\bm x \not\in \bigcup_{f \subseteq T \in \mathcal T} T = P_f = P_{\bm p}$ since $\bm p \in \mathring f$.
		Furthermore, for any $s \in T_s$ and $\bm x \in T_s$, we have $\phi_{\bm v_i^f}(\bm x) = \lambda_{s,i}(\bm x)$ by the definition of $\phi_{\bm v_i^f}(\bm x)$. This completes the proof.
	\end{proof}

	\begin{remark}
		Compared to formula \eqref{eq:varphip-convex} for vertex-convex meshes, the formulation \eqref{eq:varphip-noncovex} for arbitrary simplicial meshes appears more general and unified from the finite element perspective. However, regarding the representation by DNNs, equation \eqref{eq:varphip-convex} offers a more straightforward construction. Furthermore, it provides a more efficient implementation for Lagrange finite element functions with vertex-convex meshes, which will be discussed in the following section.
	\end{remark}

	\section{Representations of Lagrange elements using DNNs with ReLU and ReLU$^2$ activation functions}\label{sec:feDNN}
	In this section, we demonstrate the use of $\Sigma^{1,2}_{n_{1:L}}$ to represent Lagrange elements on both vertex-convex and arbitrary simplicial meshes by employing the global representation formulas presented in Theorem~\ref{thm:basisall-convex} and Theorem~\ref{thm:basisall-noncovex}, leveraging two fundamental properties of $\Sigma^{1,2}_{n_{1:L}}$.
	
	Initially, we introduce the following crucial properties of ${\rm ReLU}$ and ${\rm ReLU}^2$.
	\begin{properties}\label{pro:minmultiDNN}
		For any $(x,y) \in \mathbb R^2$, we have
		\begin{equation*}
			\min \{x,y\} = a_1{\rm ReLU}\left(W\begin{pmatrix}
				x \\ y
			\end{pmatrix}\right) \in \Sigma^1_{4} \subset \Sigma^{1,2}_4,
		\end{equation*}
		and 
		\begin{equation*}
			xy = a_2{\rm ReLU}^2\left(W\begin{pmatrix}
				x \\ y
			\end{pmatrix}\right) \in \Sigma^2_{4} \subset \Sigma^{1,2}_4,
		\end{equation*} 
		where 
		\begin{equation*} 
			W = \begin{pmatrix}
				1 & 1 \\
				-1 & -1 \\
				1 & -1 \\
				-1 & 1
			\end{pmatrix}
		\end{equation*}
		and
		\begin{equation*}
			a_1 = \frac{1}{2} \begin{pmatrix}
				1 & -1 & -1 & -1
			\end{pmatrix},\quad 
			a_2 = \frac{1}{4} \begin{pmatrix}
				1 & 1 & -1 & -1
			\end{pmatrix}.
		\end{equation*}
	\end{properties}
	\begin{proof}
		The first identity holds since
		\begin{equation*}
			\min \{x,y\} = \frac{x+y}{2} - \frac{\left|x-y\right|}{2} 
		\end{equation*}
		and
		\begin{equation*}
			x = {\rm ReLU}(x) -  {\rm ReLU}(-x)  \quad\text{and}\quad  |x| = {\rm ReLU}(x) +  {\rm ReLU}(-x).
		\end{equation*}
		The second identity holds since
		\begin{equation*}
			xy = \frac{(x+y)^2}{4} - \frac{(x-y)^2}{4}
		\end{equation*}
		and
		\begin{equation*}
			x^2 = {\rm ReLU}^2(x) +  {\rm ReLU}^2(-x).
		\end{equation*}
	\end{proof}
	
	To simplify our presentation, for any $\bm{p} \in \mathcal X_k^\ell(\mathcal T)$ with $0 \le \ell \le d$, let us define
	\begin{equation*}
		\tau(\bm{p}) = \# I_{\bm{p}},
	\end{equation*}
	as the number of simplices in the mesh $\mathcal{T}$ that contain the interpolation point $\bm{p}$. 
	
	\paragraph{Representation of basis functions on vertex-convex meshes.}
	By combining our main result in Theorem~\ref{thm:basisall-convex} and Properties~\ref{pro:minmultiDNN}, we have the following lemma for representing basis functions $\varphi_{\bm{p}}(\bm{x})$ for Lagrange elements on vertex-convex meshes.
	\begin{lemma}\label{lem:basisDNN-convex}
		For any vertex-convex simplicial mesh $\mathcal{T}$ of $\Omega$ for the $P_k$ element and $\bm{p} \in \mathcal X_k^\ell(\mathcal T)$ for some $0 \le \ell \le d$, there exists
		\begin{equation}
			\widetilde \varphi_{\bm{p}}(\bm{x}) \in \Sigma^{1,2}_{n_{1:L}},
		\end{equation}
		where 
		\begin{equation}
			L \le \left\lceil \log_2(\tau(\bm{p})) \right\rceil + 
			\max\left\{\left\lceil \log_2(\ell+1) \right\rceil, \left\lceil \log_2(k-\ell-1) \right\rceil\right\} + 3
		\end{equation}
		and
		\begin{equation}
			n_{j} \le 2^{3-j}\max\{\ell+1,k-\ell-1\}\tau(\bm{p}).
		\end{equation}
		for all $j=1,\ldots,L$ such that $\widetilde \varphi_{\bm{p}}(\bm{x}) = \varphi_{\bm{p}}(\bm{x})$ for all $\bm x \in \Omega$. Furthermore, $\widetilde \varphi_{\bm p}(\bm x) =  \varphi_{\bm p}(\bm x)$ for any $\bm x \in \mathbb R^d$ if $\bm p \in \mathring \Omega$. 
	\end{lemma}
	\begin{proof}
		As demonstrated in \eqref{eq:varphip-convex} from Theorem~\ref{thm:basisall-convex}, we have
		\begin{equation}
			\varphi_{\bm{p}}(\bm{x}) = \frac{1}{\bm \alpha !} \min_{i' \in I_{\bm{p}}}\left\{\prod_{i=0}^\ell k{\rm ReLU}(\lambda_{i',i}(\bm{x})) \right\} \times 
			\prod_{i=0}^\ell \prod_{j=1}^{\alpha_i-1}\left(k\min_{i' \in I_{\bm{p}}}\left\{\lambda_{i',i} (\bm{x}) \right\} -j \right).
		\end{equation}
		This formulation allows us to construct a DNN where $\lambda_{i',i}(\bm{x})$ is the input of the first layer with $\tau(\bm{p}) \times (\ell+1)$ neurons. Subsequent operations such as $\prod_{i=0}^\ell k{\rm ReLU}\left(\lambda_{i',i}(\bm{x})\right)$, $\min_{i' \in I_{\bm{p}}}\left\{\lambda_{i',i} (\bm{x}) \right\}$, $\min_{i' \in I_{\bm{p}}}\left\{\cdot\right\}$, and $\prod_{i=0}^\ell\prod_{j=1}^{\alpha_i-1} \left(\cdot\right)$ can be implemented utilizing Property~\ref{pro:minmultiDNN}. Specifically, we have
		\begin{equation}
			\min_{i' \in I_{\bm{p}}}\left\{\prod_{i=0}^\ell k {\rm ReLU}\left(\lambda_{i',i}(\bm{x})\right) \right\} \in \Sigma^{1,2}_{n_{1:L_1}},
		\end{equation}
		where $L_1 \le \left\lceil \log_2(\tau(\bm{p})) \right\rceil + \left\lceil \log_2(\ell+1) \right\rceil + 1$ and $n_{j} \le 2^{3-j}(k-\ell-1)\tau(\bm{p})$, and
		\begin{equation}
			\prod_{i=0}^\ell\prod_{j=1}^{\alpha_i-1} \min_{i' \in I_{\bm{p}}}\left\{k\lambda_{i',i} (\bm{x}) -j\right\} \in \Sigma^{1,2}_{n_{1:L_2}},
		\end{equation}
		where $L_2 \le \left\lceil \log_2(\tau(\bm{p})) \right\rceil + \left\lceil \log_2(k-\ell-1) \right\rceil + 1$ and $n_{j} \le 2^{3-j}(\ell+1)\tau(\bm{p})$.
		By incorporating an additional hidden layer to combine (multiply) these two components, we establish that
		\begin{equation}
			\widetilde \varphi_{\bm{p}}(\bm{x}) \in \Sigma^{1,2}_{n_{1:L}},
		\end{equation}
		where $L \le \left\lceil \log_2(\tau(\bm{p})) \right\rceil  + 
		\max\left\{\left\lceil \log_2(\ell+1) \right\rceil, \left\lceil \log_2(k-\ell-1) \right\rceil\right\} + 3$ and
		$n_{j} \le 2^{3-j}\max\{\ell+1,k-\ell-1\}\tau(\bm{p})$.
	\end{proof}

	\paragraph{Representation of basis functions on arbitrary simplicial meshes.}
	For general simplicial mesh $\mathcal T$ that may not be vertex-convex, we first introduce how to use $\Sigma^{1,2}_{n_{1:L}}$ to represent the basis functions for linear finite elements on $\mathcal T$.
	
	\begin{theorem}[\cite{longo2023rham}, Theorem 4.3]\label{thm:linearbasis-nonconvex}
		For any simplicial mesh $\mathcal{T}$ of $\Omega$, $\bm v \in \mathcal X_k^0(\mathcal T) = \Delta_0(\mathcal T)$, and $\phi_{\bm v}$ the basis function of linear finite element at $\bm v$, there exist 
		\begin{equation}
			\widetilde \phi_{\bm v}(\bm x) \in \Sigma_{n_{1:L}}^1 \subset \Sigma_{n_{1:L}}^{1,2}
		\end{equation}
		where $L \le \left\lceil \log_2(\tau(\bm{v})) \right\rceil  + \left\lceil \log_2(d+1) \right\rceil + 7$ and $n_{j} \le 2^{3-j} (d+1)\tau(\bm{v})$ for all $j=1,\ldots,L$ such that $\widetilde \phi_{\bm v}(\bm x) = \phi_{\bm v}(\bm x)$ for all $\bm x \in \Omega$. Furthermore, $\widetilde \phi_{\bm v}(\bm x) =  \phi_{\bm v}(\bm x)$ for any $\bm x \in \mathbb R^d$ if $\bm v \in \mathring \Omega$. 
	\end{theorem}
	
	\begin{remark}
		Theorem~\ref{thm:linearbasis-nonconvex} is based on a main theorem in \cite{he2020relu} about the representation of the linear basis function $\phi_{\bm{v}}(\bm{x})$ by ReLU DNNs when $P_{\bm{v}}$ is convex for $\bm{v} \in \mathcal{X}_k^0(\mathcal{T})$. The key step is the overlapping decomposition $P_{\bm{v}} = \bigcup_{i \in I_{\bm{v}}} P_{\bm{v},i}$, where $\bm{v} \in \mathring{P}_{\bm{v},i}$ and $P_{\bm{v},i}$ is a simplex consisting of $d+1$ simplices. That is, we can understand $P_{\bm{v},i}$ as local vertex-convex meshes of $\bm{v}$ for all $i \in I_{\bm{v}}$.
		Then, by using the result in \cite{he2020relu}, we can construct ReLU DNNs $\phi_{\bm{v},i}(\bm{x})$ for all $i \in I_{\bm{v}}$ such that $\phi_{\bm{v},i}(\bm{x})$ is the linear basis function in terms of the local vertex-convex meshes $P_{\bm{v},i}$. Finally, one can have $\phi_{\bm{v}}(\bm{x}) = \max_{i \in I_{\bm{v}}}\left\{ \phi_{\bm{v},i}(\bm{x})\right\}$. More details can be found in~\cite{longo2023rham}.
	\end{remark}

	By combining our main result in Theorem~\ref{thm:basisall-noncovex} with Theorem~\ref{thm:linearbasis-nonconvex} and Properties~\ref{pro:minmultiDNN}, we have the following lemma for representing basis functions $\varphi_{\bm{p}}(\bm{x})$ for Lagrange elements on arbitrary simplicial meshes.
	\begin{lemma}\label{lem:basisDNN-nonconvex}
		For any simplicial mesh $\mathcal{T}$ of $\Omega$ for the $P_k$ element and $\bm{p} \in \mathcal X_k^\ell(\mathcal T)$ for some $0 \le \ell \le d$, let us assume $f = {\rm Conv}(\bm v_0^f, \bm v_1^f, \ldots, \bm v_\ell^f)$ and $\bm p \in \mathring f$ and $\phi_{\bm v_i^f}$. Then, there exists
		\begin{equation}
			\widetilde \varphi_{\bm{p}}(\bm{x}) \in \Sigma^{1,2}_{n_{1:L}},
		\end{equation}
		where 
		\begin{equation}
			L \le \max_{i=0,1,\ldots,\ell}\left\lceil \log_2(\tau(\bm{v}_i^f)) \right\rceil + \left\lceil \log_2(d+1) \right\rceil + \left\lceil \log_2(k) + 7\right\rceil
		\end{equation}
		and
		\begin{equation}
			n_{j} \le 2^{3-j} (d+1) \sum_{i=0,1,\ldots,\ell} \tau(\bm{v}_i^f),
		\end{equation}
		for all $j=1,\ldots,L$ such that $\widetilde \varphi_{\bm{p}}(\bm{x}) = \varphi_{\bm{p}}(\bm{x})$ for all $\bm x \in \Omega$. Furthermore, $\widetilde \varphi_{\bm p}(\bm x) =  \varphi_{\bm p}(\bm x)$ for any $\bm x \in \mathbb R^d$ if $\bm p \in \mathring \Omega$. 
	\end{lemma}
	\begin{proof}
		As demonstrated in \eqref{eq:varphip-noncovex} from Theorem~\ref{thm:basisall-noncovex}, we have
		\begin{equation*}
			\varphi_{\bm{p}}(\bm{x}) = \frac{1}{\bm{\alpha} !} 
			\prod_{i=0}^\ell\prod_{j=0}^{\alpha_i -1}\left(k \phi_{\bm{v}_i^f}(\bm{x}) -j \right) \quad \forall \bm{x} \in \Omega.
		\end{equation*}
		Given that $\phi_{\bm{v}_i^f}$ can be represented by $\Sigma_{n_{1:L}}^1 \subset \Sigma_{n_{1:L}}^{1,2}$ and $\Sigma^{1,2}_{n_{1:L}}$ can handle the multiplication operation, the above formulation allows us to construct a DNN function in $\Sigma^{1,2}_{n_{1:L}}$ to recover $\varphi_{\bm{p}}(\bm{x})$ by first representing $\phi_{\bm{v}_i^f}(\bm{x})$. 
		According to Theorem~\ref{thm:linearbasis-nonconvex}, $\phi_{\bm{v}_i^f}(\bm{x})$ can be reproduced with $\left\lceil \log_2(\tau(\bm{v})) \right\rceil  + \left\lceil \log_2(d+1) \right\rceil + 7$ hidden layers and $2^{3-j} (d+1)\tau(\bm{v})$ neurons in the $j$-th hidden layer. 
		Noting that there are extra $\sum_{j=0}^\ell \alpha_i = k$ multiplications from $\phi_{\bm{v}_i^f}(\bm{x})$ to $\varphi_{\bm{p}}(\bm{x})$, we require only an additional $\left\lceil \log_2(k) \right\rceil$ hidden layers to recover these multiplication operations and to parallelly concatenate hidden layers of different $\phi_{\bm{v}_i^f}(\bm{x})$. Thus, we can establish bounds for $L$ and $n_j$.
	\end{proof}
	
	\paragraph{Representation of Lagrange finite element functions using $\Sigma_{n_{1:L}}^{1,2}$.}
	With the representation results for basis functions in place, we can state our main theorem as follows.
	\begin{theorem}\label{thm:uhdnn}
		For any $P_k$ Lagrange finite element function $u_h(\bm{x})$ on a 
		simplicial mesh $\mathcal{T}$ of $\Omega$, there is
		\begin{equation}
			\widetilde u_h(\bm{x}) \in \Sigma_{n_{1:L}}^{1,2},
		\end{equation}
		such that $\widetilde u_h(\bm{x}) = u_h(\bm x)$ for any $\bm x \in \Omega$.
		More precisely, we have the following bounds of $\Sigma_{n_{1:L}}^{1,2}$:
		\begin{description}    
			\item[If $\mathcal T$ is vertex-convex,] 
			\begin{equation}
				L \le \sum_{\ell=0}^{d}\eta_\ell \left( \left\lceil \log_2(\tau_\ell) \right\rceil + \max\left\{\left\lceil \log_2(\ell+1) \right\rceil, \left\lceil \log_2(k-\ell-1) \right\rceil\right\} + 3\right)
			\end{equation}
			and 
			\begin{equation}
				\max_{j} n_j \le 4 \max_\ell \left\{ \max\{\ell+1,k-\ell-1\}\tau_\ell \right\}  + 2d + 1;
			\end{equation}
			\item[If $\mathcal T$ is arbitrary,]
			\begin{equation}
				L \le \left(\left\lceil \log_2(\tau_0) \right\rceil + \left\lceil \log_2(d+1) \right\rceil + \left\lceil \log_2(k) \right\rceil + 7\right) \sum_{\ell=0}^d \eta_\ell
			\end{equation}
			and
			\begin{equation}
				\max_j n_j \le 4 (d+1)\min\{d, k+1\} \tau_0 + 2d + 1,
			\end{equation}
		\end{description}
		where $\tau_\ell := \max_{\bm{p} \in \mathcal X_k^\ell(\mathcal{T})} \tau(\bm{p})$ and $\eta_\ell = \# \mathcal X_k^\ell(\mathcal{T})$ for all $\ell=0,1,\ldots,d$.
	\end{theorem}
	
	To simplify the proof, we introduce the following revised neural network architecture:
	\begin{equation}\label{eq:defdnn_new}
		\begin{cases}
			&\widetilde{u}^0(\bm{x}) =  \bm{x},\\
			&\left[\widetilde{u}^j (\bm{x})\right]_i = \sigma^j_i \left(\left[\widetilde W^\ell \left(\widetilde{u}^{j-1}(\bm{x}), \bm{x}\right) + \widetilde b^j\right]_i \right), \quad i=1, \cdots, n_j,~j=1, \cdots, L,\\
			&\widetilde{u}(\bm{x}) = \widetilde W^{L+1} \left(\widetilde{u}^0, \widetilde{u}^1, \cdots, \widetilde{u}^L\right) + \widetilde b^{L+1},
		\end{cases}
	\end{equation}
	and denote it as $\widetilde{\Sigma}_{n_{1:L}}^{1,2}$ when $\sigma_i^\ell$ shares the same definition as in \eqref{eq:sigma}. This type of network architecture is widely studied in the literature, such as in \cite{daubechies2022nonlinear,he2022relu}, where $\sigma_i^j = {\rm ReLU}$. Here, we generalize this structure to the learnable ${\rm ReLU}$-${\rm ReLU}^2$ activation function while preserving the following important properties.
	\begin{properties}\label{prop:tildeSigma}
		For $\widetilde{\Sigma}_{n_{1:L}}^{1,2}$ on $\mathbb{R}^d$, we have:
		\begin{enumerate}
			\item 
			\begin{equation}
				\widetilde{\Sigma}_{n_{1:L_1}}^{1,2} + \widetilde{\Sigma}_{m_{1:L_2}}^{1,2} \subseteq 
				\widetilde{\Sigma}_{\left(n_{1:L_1}, m_{1:L_2}\right)}^{1,2},
			\end{equation}
			where $\left(n_{1:L_1}, m_{1:L_2}\right)$ equals $\left(n_{1}, \cdots, n_{L_1}, m_1, \cdots, m_{L_2}\right)$;
			\item 
			\begin{equation}
				\Sigma_{\widetilde{n}_{1:L}}^{1,2} \subseteq  \widetilde{\Sigma}_{\widetilde{n}_{1:L}}^{1,2} \subseteq \Sigma_{n_{1:L}}^{1,2},
			\end{equation}
			provided $n_j \ge \widetilde{n}_j + 2d + 1$ for all $j$.
		\end{enumerate}
	\end{properties}
	\begin{proof}
		The first property is a direct consequence of the definition of $\widetilde{\Sigma}_{n_{1:L}}^{1,2}$. The proof follows the same lines as in \cite{he2022relu}. The second property is valid because $\sigma_i^j$ can replicate the identity function, which can be demonstrated using a similar argument to that in \cite{he2022relu}.
	\end{proof}

	\paragraph{Proof of Themorem~\ref{thm:uhdnn}}
	\begin{proof}
		For any $u_h(\bm x)$, it can be written as
		\begin{equation}
			u_h(\bm x) = \sum_{\bm p \in \mathcal X_k(\mathcal T)} u_{h}(\bm p) \varphi_{\bm p}(\bm x) = \sum_{\ell=0}^{d}\sum_{\bm p \in \mathcal X_k^\ell(\mathcal T)} u_{h}(\bm p) \varphi_{\bm p}(\bm x).
		\end{equation}
		Given Properties~\ref{prop:tildeSigma} and the estimates in Lemma~\ref{lem:basisDNN-convex} and Lemma~\ref{lem:basisDNN-nonconvex}, we have
		\begin{equation}
			\widetilde u_h(\bm x) =  \sum_{\ell=0}^{d}\sum_{\bm p \in \mathcal X_k^\ell(\mathcal T)} u_{h}(\bm p) \widetilde \varphi_{\bm p}(\bm x) \in \widetilde \Sigma_{\tilde n_{1:L}}^{1,2} \subset \Sigma_{n_{1:L}}^{1,2}
		\end{equation}
		where $L$ and $n_j$ are bounded as above.
	\end{proof}
	
	\begin{remark}
		For simplicial meshes, we observe that $\tau_\ell \geq \tau_{\ell'}$ for any $0 \leq \ell \leq \ell' \leq d$. In particular, $\tau_0$ is typically the largest and often scales as the exponent (or even factorial) of $d$. 
		Thus, we note that the depths and widths required for vertex-convex meshes are significantly smaller than those for arbitrary simplicial meshes, as demonstrated in Theorem~\ref{thm:uhdnn}.
	\end{remark}

	\section{An application of the representation result for approximation properties of $\Sigma_{n_{1:L}}^{1,2}$}\label{sec:application}
	In this section, we present an approximation result for $\Sigma_{n_{1:L}}^{1,2}$ utilizing representation theory on a special vertex-convex mesh. We assume that $u \in W^{q+1,p}(\Omega)$ for some $k \geq 1$ and $p \in [2,\infty]$ with $\Omega = [0,1]^d$. The space $W^{k+1,p}(\Omega)$ denotes the standard Sobolev space~\cite{adams2003sobolev}. We establish the following approximation property of $\Sigma_{n_{1:L}}^{1,2}$.
	\begin{theorem}\label{thm:approx}
		For any function $u \in W^{k+1,p}(\Omega)$ with $p \in [2,\infty]$ and $N>0$, we obtain the approximation rate
		\begin{equation}\label{eq:approx}
			\inf_{\widetilde{u} \in \Sigma_{n_{1:L}}^{1,2}} \|u - \widetilde{u}\|_{W^{s,p}(\Omega)} = \|u\|_{W^{k+1,p}(\Omega)} \mathcal{O}\left( N^{-\frac{k+1-s}{d}}\right),
		\end{equation}
		for any $s = 0, 1$. Here, $L \leq C_1 \log(k) k^d N$ and $n_{j} \leq C_2 k$, where $C_1$ and $C_2$ depend only on $d$.
	\end{theorem}
	
	\paragraph{Uniform simplex mesh on $[0,1]^d$ with size $h = \frac{1}{m}$}
	To prove Theorem~\ref{thm:approx}, we introduce some special simplices with size $h=\frac{1}{m}$ for some $m \in \mathbb{N}^+$ in $[0,1]^d$ as
	\begin{equation}
		\Delta_{\bm{v}, \rho}^h := \left\{ \bm{x} \in [0,1]^d : 0 \leq x_{\rho(1)} - v_{\rho(1)} \leq \cdots \leq x_{\rho(d)} - v_{\rho(d)} \leq h \right\}
	\end{equation}
	for any $\bm{v} \in \bm{Z}_h^d := \left\{0, h, 2h, \cdots, (m-1)h, 1\right\}^d$ and any permutation $\rho$ of $\{1,2,\cdots,d\}$. The corresponding uniform simplicial mesh $\mathcal{T}_h^d$ with size $h$ on $[0,1]^d$ is
	\begin{equation}\label{eq:mesh}
		\mathcal{T}_h^d:= \left\{ \Delta_{\bm{v}, \rho}^h : \bm{v} \in \bm{Z}_h^d, \rho \text{ is a permutation of } \{1,2,\cdots,d\}\right\} = [0,1]^d.
	\end{equation}
	This particular mesh, $\mathcal{T}_h^d$, is known as the Freudenthal triangulation~\cite{freudenthal1942simplizialzerlegungen} and is also referred to as the Kuhn partition, as described in the paper~\cite{kuhn1960some}. This triangulation can also result from decomposing $[0,1]^d$ with hyperplanes defined by $x_i - x_j = kh$ and $x_i = sh$ for any $1 \leq i,j \leq d$ and $k,s = 0, \cdots, m$. Hence, the set of all vertices of $\mathcal{T}_h^d$ is $\bm{Z}_h^d$. As demonstrated in~\cite{yarotsky2018optimal}, $\text{supp}(\varphi_{\bm{v}})$ is convex for any vertex $\bm{v} \in \bm{Z}_h^d = \mathcal X_k^0(\mathcal T_h^d)$. That is, $\mathcal T_h^d$ is vertex-convex.

	\paragraph{Proof of Theorem~\ref{thm:approx}}
	\begin{proof}
		Let $V_h^k$ denote the space of $k$-th order Lagrange elements on the uniform mesh $\mathcal{T}_h^d$ of $[0,1]^d$ with size $h = \frac{1}{m}$. Given the estimate for the vertex-convex mesh in Theorem~\ref{thm:uhdnn} and the self-similarity of $\mathcal{T}_h^d$, we observe that $\tau_\ell$ for all $\ell=0,1,\ldots,d$ are determined only by the dimension $d$ and is independent of the mesh size $h$. Hence, we have 
		\begin{equation}
			\max_{j} n_{j} \leq C_W k,
		\end{equation}
		where $C_W$ is a constant depending only on $d$.
		Additionally, we find that
		\begin{equation}
			\begin{aligned}
				&\left\lceil \log_2(\tau_\ell) \right\rceil + 
				\max_{1\leq \ell\leq \min\{k,d+1\}}\left\{\left\lceil \log_2(\ell) \right\rceil, \left\lceil \log_2(k-\ell-1) \right\rceil\right\} + 3\\
				\leq &\left\lceil \log_2(\tau_\ell) \right\rceil + \max\left\{ \left\lceil \log_2(d+1) \right\rceil, \left\lceil\log_2(k)\right\rceil \right\} \leq C_L \log(k) + 3,
			\end{aligned}
		\end{equation}
		where $C_L$ is a constant dependent only on $d$. Consequently, we obtain
		\begin{equation}
			L \leq  C_L \log(k) \sum_{\ell=0}^{d} \eta_\ell,
		\end{equation}
		noting that $\sum_{\ell=0}^{d} \eta_\ell$ is the total number of degrees of freedom for $V_h^k$, which can be bounded by
		\begin{equation}
			\sum_{\ell=0}^{d} \eta_\ell = \text{dim}(V_h^k) \leq C_T k^d m^d,
		\end{equation}
		where $C_T$ depends only on $d$. As a result, we have
		\begin{equation}
			V_h^k \subseteq \Sigma_{n_{1:L}}^{1,2},
		\end{equation}
		with $L \leq C_1 \log(k) k^d m^d$ and $n_{j} \leq C_2 k$, where $C_1$ and $C_2$ are constants depending only on $d$.
		
		Utilizing the classical approximation results of $V_h^k$ for functions in $W^{k+1,p}(\Omega)$, as shown in~\cite{ciarlet2002finite,wang2013finite}, we have 
		\begin{equation}
			\inf_{\widetilde{u} \in  \Sigma_{n_{1:L}}^{1,2}} \|u - \widetilde{u}\|_{W^{k,p}(\Omega)} \leq   \inf_{\widetilde{u} \in V_h^k} \|u - \widetilde{u}\|_{W^{s,p}(\Omega)} = \mathcal{O}(\|u\|_{W^{k+1,p}} h^{k+1-s}),
		\end{equation}
		for any $s=0,1$ and $p\in [2,\infty]$, where 
		$L \leq C_1 \log(k) k^d m^d$ and $n_{j} \leq C_2 k$, and $C_1$ and $C_2$ depend only on $d$. The proof is completed by noting that $h=\frac{1}{m}$ and substituting $m^d$ with $N$.
	\end{proof}
	
	\begin{remark}
		In the approximation estimate given by \eqref{eq:approx}, the case for $s=0$ and $k=1$, that is, the $L^p$ estimate for linear finite element interpolation, has already been established in \cite{he2020relu,yarotsky2018optimal}. By incorporating bit-extraction techniques, the $L^p$ approximation rate for ReLU DNNs (i.e., $k=1$) can be enhanced to $N^{-\frac{2}{d}}$ for functions in $C(\Omega)$, as shown in \cite{yarotsky2018optimal,shen2020deep}, and to $N^{-\frac{2k}{d}}$ for functions in $C^{k}(\Omega)$, as demonstrated in \cite{lu2021deep}. Notably, the $W^{s,p}$ estimate for $s=1$ is first introduced in this work.
	\end{remark}

	\section{Concluding remarks}\label{sec:conclusions}
	In conclusion, this paper validates the capability of $\Sigma_{n_{1:L}}^{1,2}$ to represent Lagrange finite element functions across dimensions and orders, leveraging two novel formulations of global basis functions on both vertex-convex and arbitrary simplicial meshes, geometric decomposition, and insights from high-dimensional simplicial meshes, barycentric functions, and multiplication of linear basis functions. This framework not only clarifies the basis function representations but also encompasses the full scope of Lagrange finite element functions with efficient parameterizations. Moreover, our investigation into the uniform simplex mesh and the associated approximation properties has yielded a straightforward proof for the approximation result for $\Sigma_{n_{1:L}}^{1,2}$.
	
	Additionally, this work opens the door to several interesting mathematical inquiries in the domains of deep neural networks and finite element methods. A pertinent question is whether $\Sigma_{n_{1:L}}^{1,2}$ with $\mathcal{O}(N)$ parameters can effectively represent Lagrange finite element functions with $N$ degrees of freedom. Efforts to improve the approximation rate could also involve combining the representation results with bit-extraction techniques, as shown for linear finite elements in~\cite{yarotsky2018optimal}. Furthermore, the expressiveness of $\Sigma_{n_{1:L}}^{1,2}$ could potentially be generalized to encompass a wider array of finite elements, such as those from finite element exterior calculus~\cite{arnold2006finite, arnold2010finite,longo2023rham}, elements with higher smoothness~\cite{hu2023construction, chen2021geometric}, $h$-$p$ adaptive elements~\cite{babuvska1994p}, and even non-conforming elements~\cite{wang2013finite}. On the other hand, in contrast to traditional methods that rely on domain partitioning and function expression via bases and degrees of freedom, our approach employs deep neural networks for finite element function representation within a cohesive framework. This innovation could revolutionize finite element implementation and inspire new ``adaptive" methods for solving PDEs, as shown in the study of linear finite elements in one dimension in \cite{he2020relu}.

	\newpage
	\bibliographystyle{plain}
	\bibliography{dnn_fem.bib}
	
\end{document}